\documentclass{amsart}

\usepackage{color}
\usepackage{amsmath}
\usepackage{amscd,amsfonts,amssymb}
\usepackage[all]{xy}
\usepackage{mathtools}
\usepackage[colorlinks=true]{hyperref}
\usepackage{tikz}
\usepackage{tikz-cd}

\def\R{\mathbb{R}}

\def\Z{\mathbb{Z}}

\def\deg{{\mathrm{deg}}}

\def\shfL{\mathcal{L}}

\def\ovl{\overline}

\usepackage[normalem]{ulem}
\definecolor{mred}{rgb}{0.83, 0.0, 0.0}
\definecolor{darkspringgreen}{rgb}{0.09, 0.45, 0.27}
\definecolor{ruby}{rgb}{0.88, 0.07, 0.37}
\def\colorsout#1{\bgroup\markoverwith{\textcolor{#1}{\rule[0.5ex]{2pt}{0.7pt}}}\ULon} 
\def\coloruline#1{\bgroup\markoverwith{\textcolor{#1}{\rule[-0.5ex]{2pt}{0.7pt}}}\ULon} 

\theoremstyle{plain}

\numberwithin{equation}{section}

\makeatletter
\newcommand\tint{\mathop{\mathpalette\tb@int{t}}\!\int}
\newcommand\bint{\mathop{\mathpalette\tb@int{b}}\!\int}
\newcommand\tb@int[2]{%
  \sbox\z@{$\m@th#1\int$}%
  \if#2t%
    \rlap{\hbox to\wd\z@{%
      \hfil
      \vrule width .35em height \dimexpr\ht\z@+1.4pt\relax depth -\dimexpr\ht\z@+1pt\relax
      \kern.05em 
    }}
  \else
    \rlap{\hbox to\wd\z@{%
      \vrule width .35em height -\dimexpr\dp\z@+1pt\relax depth \dimexpr\dp\z@+1.4pt\relax
      \hfil
    }}
  \fi
}
\makeatother

\usepackage{etoolbox}
\makeatletter
\newcommand*\suppresschapternumber{%
  \let\@makechapterhead\@makeschapterhead
  \patchcmd{\@chapter}
    {\protect\numberline{\thechapter}}
    {}
    {}{}%
}
\newcommand*\removedotbetweenchapterandsection{%
  \renewcommand\thesection{\thechapter\@arabic\c@section}%
}
\makeatother

\newcommand{\gl}{{\rm GL}}
\newcommand{\spec}{{\rm Spec~}}

\long\def\comment#1{}

\newtheorem{thm}{Theorem}[section] 
\newtheorem{prop}[thm]{Proposition}
\newtheorem{lem}[thm]{Lemma}
\newtheorem{defn}[thm]{Definition}
\newtheorem{cor}[thm]{Corollary}
\newtheorem{warn}[thm]{Warning}

\theoremstyle{definition} 

\theoremstyle{remark} 
\newtheorem{rem}[thm]{Remark}

\begin{document}

\title{Geometric height on flag varieties in positive characteristic}
\author{Yue Chen}
\address{Department of Mathematics, Cornell University, 14850, Ithaca, New York, U.S.A}
\email{yc2949@cornell.edu}

\author{Haoyang Yuan}
\address{School of Mathematics, Nanjing University, Nanjing 210093, China}
\email{zgqyhy@163.com}

\begin{abstract}
Let $k$ be an algebraically closed field of characteristic $p\neq 0$. Let $G$ be a connected reductive group over $k$, $P \subseteq G$ be a parabolic subgroup and $\lambda: P \longrightarrow \mathbb G_m$ be a strictly anti-dominant character.  Let $C$ be a projective smooth curve over $k$ with function field $K=k(C)$ and   $F$ be a principal $G$-bundle on $C$. We compute explicitly the height filtration and successive minima of the height function $h_{\mathcal{L}_\lambda}: X(\overline{K}) \longrightarrow \mathbb{R}$  over the flag variety $X=(F/P)_K$ through relatively ample line bundle the $\mathcal{L}_\lambda=F \times_P k_\lambda$ on $F/P$.
\end{abstract}

\maketitle
\setcounter{tocdepth}{1}
\tableofcontents

\section{Introduction}
\subsection{Height filtration and successive minima}
Let $K$ be either a number field or $K=k(C)$ where $C$ is a projective smooth curve over a field $k$. Let $X$ be a projective variety of dimension $d$ over $K$ and $\overline{L}$ be an adelic line bundle on $X$. These datum induce an Arakelov height function $h_{\overline{L}}$ on $X$ (see \cite[\textsection 9]{Yuan_iccm2010} for a survey). 

A typical case is the \textit{geometric height}, which is the one we concern in this article. Here we give an equivalent definition.

In our case $K=k(C)$, suppose there is a projective flat morphism $\mathcal{X} \longrightarrow C$ with the generic fiber $X \longrightarrow \operatorname{Spec}(K)$ and 
 a line bundle $\mathcal{L}$   on $\mathcal{X}$. The data $(\mathcal{X},\mathcal{L})$ define an adelic line bundle $\ovl L$ and the height function $h_{\ovl L}$ is given by
        \begin{flalign*}
            \quad \quad \quad  &
                h_{\overline{L}}: X(\overline{K}) \longrightarrow \mathbb{Q}, \; x \longmapsto \frac{\mathcal{L} \cdot \overline{\{x\}}}{\deg(x)} \;\; \text{where $\overline{\{x\}}$ is the closure of $x$ in $\mathcal{X}$}.
            &&
        \end{flalign*}
We also denote this by $h_\shfL$ if there is no ambiguity.  For a different choice $(\mathcal{X},\mathcal{L})$, the height functions differ by a bounded function, see \cite{Yuan_iccm2010}.

For any $t \in \mathbb{R}$, let $Z_t \subseteq X$ be the Zariski closure of the set $\big\{ x \in X(\overline{K}): h_{\overline{L}}(x) < t \big\}$. We call $\big\{ Z_t: t \in \mathbb{R} \big\}$ the \emph{height filtration}, and call its jumping points the \emph{successive minima}. Note that this definition of successive minima is slightly different from Zhang \cite{Zhang_smallPoints}. Zhang considers only dimension jumps $e_i=\inf\big\{ t: \dim Z_t \geq d-i+1 \big\}$, $i=1,\dots,d+1$. 


In \cite{fan2024arakelovgeometryflagvarieties}, Fan-Luo-Qu provides a new case where height filtration can be explicitly computed, which is the geometric height on flag varieties under the assumption ${\rm char}(k)=0$. Following their idea, we complete this calculation by examining this case in positive characteristic. 

Following previous notations, let $G$ be a connected reductive group over $k$, $P \subseteq G$ be a parabolic subgroup, and $\lambda: P \longrightarrow \mathbb{G}_m$ be a strictly antidominant character. Let $F$ be a principal $G$-bundle over $C$ and $\mathcal{X}=F/P$ with generic fiber $X=(F/P)_K$. Let $\mathcal{L}_\lambda = F \times_P k_\lambda$ and $h_{\mathcal{L}_\lambda}: X(\overline{K}) \longrightarrow \mathbb{R}$ be the induced height function by the following diagram. 

\[ \xymatrix{
X=(F/P)_K \ar[d] \ar[r] & \mathcal{X}=F/P \ar[d] \\
\operatorname{Spec}(K) \ar[r] & C
} \]

Denote $F_Q$ to be the canonical reduction of $F$, which is a principal $Q$-bundle for a parabolic subgroup $Q\subseteq G$, and also $\deg(F_Q) \in X(T)^\vee_\mathbb{Q}$ to be the induced cocharacter. Let $W,W_P$ and $W_Q$ be the Weyl group of $G,L_P$ and $L_Q$ correspondingly. For $w \in W_Q \backslash W/W_P$, let $C_w=(F_Q \times_Q QwP/P)_K \subseteq X$ be the corresponding Schubert cell. The numbers $\langle \deg (F_{Q }), w\lambda \rangle$ are well-defined. One main result in \cite{fan2024arakelovgeometryflagvarieties} computes the height filtration in characteristic zero.

\begin{thm}[Theorem 2.1, \cite{fan2024arakelovgeometryflagvarieties}] \label{qu}
Assume $k$ has characteristic zero. For any $t \in \mathbb{R}$, let $Z_t \subseteq X$ be the Zariski closure of the set $\big\{ x \in X(\overline{K}): h_{\mathcal{L}_\lambda}(x) < t \big\}$. Then
\begin{flalign*}
    \quad \quad \quad &
    Z_t = X \Bigg\backslash \coprod_{\langle \deg (F_{Q }), w\lambda \rangle\geq t} C_w = \coprod_{\langle \deg (F_{Q }), w\lambda \rangle<t} C_w.
    &&
\end{flalign*} \label{intro_ thm of ht filtration}
\end{thm}

From now on, we assume $k$ has positive characteristic unless other description is given. We compute the height filtration in this case. We show that this result does not hold in characteristic $p$ unless one puts the following assumption on $F$, namely $F$ admits a strongly canonical reduction. 

\begin{defn}[Definition \ref{strongly}]
    We say a reduction $F_Q$ to a parabolic subgroup $Q$ is strongly canonical if for any non-constant finite morphism $f: C' \to C$, the pullback $f^*(F_Q)$ is the canonical reduction of $f^*F$.
\end{defn}

\begin{thm}[Theorem \ref{charp}]
    The result in Theorem \ref{qu} holds under the assumption that $F$ admit a strongly canonical reduction. 
\end{thm}

Canonical reduction is preseved by pullbacks along separated morphisms, which means the canonical reduction is strong in characteristic zero. For arbitrary \(G\)-bundles that may not admit a strongly canonical reduction, we have the following rough form of our another main result. 

\begin{thm}\label{rough}
    The height filtration of $X$ is given by successively deleting some Frobenius twists of Schubert cells. 
\end{thm}

More explicitly, a theorem by Langer \cite{Langer05} shows for a principal $G$-bundle $F$ the strongly canonical reduction exists for $({\rm Fr}^n)^* F$, when $n$ is sufficiently large. Therefore, consider the following Cartesian diagram for such $n$, 
 \[\begin{tikzcd}
	(\operatorname{Fr_C}^n)^* F /P & F/P \\
	C & C
	\arrow["\phi",from=1-1, to=1-2]
	\arrow[from=1-1, to=2-1]
	\arrow[from=1-2, to=2-2]
	\arrow["\operatorname{Fr}_C^n" ,from=2-1, to=2-2]
\end{tikzcd}\]
with its generic fiber 
 \[\begin{tikzcd}
	\Tilde{X}=((\operatorname{Fr_K}^n)^* F /P)_K & X=(F/P)_K \\
	\spec K & \spec K
	\arrow["\phi_K",from=1-1, to=1-2]
	\arrow[from=1-1, to=2-1]
	\arrow[from=1-2, to=2-2]
	\arrow["\operatorname{Fr}_C^n" ,from=2-1, to=2-2]
\end{tikzcd}\]
where $\operatorname{Fr}_C$ is the absolute Frobenius on $C$.

\begin{thm}[Theorem \ref{twist}]
    Suppose $n$ is large enough in the above sense, the height filtration of $h_{\mathcal{L}_{\lambda}}$ on $X$ is given by the image of the height filtration of the height filtration of $\Tilde{X}$ along the homeomorphism $\phi_K$, the successive minima of $X$ are $\frac{1}{p^n}$ of the succesive minima of $\Tilde{X}$. 
\end{thm}

\begin{rem}
    Here the height filtration of $\Tilde{X}$ is known and computed by Theorem \ref{charp}. Thus, we say the height filtration of $X$ is given by successively deleting some Frobenius twist of Schubert cells as stated in Theorem \ref{rough}. 
\end{rem}

\subsection{Toy example: projective spaces} \label{example_grassmann}
Let $k$ be an algebraically closed field of positive characteristic and $C$ be a curve over $k$ with function field $K=k(C)$. Let $E$ be a vector bundle of rank $n$ on $C$. We take $\mathcal X=\mathbb P(E)$, $\mathcal L$ to be the relative ample line bundle $\mathcal O_{\mathbb P(E)}(1)$ and $X=\mathbb P(V)$ where $V$ is the generic fiber of $E$. A special case of a theorem by Ballaÿ \cite{Ballay21} shows that: 
Let $\zeta_{ess}:=\inf_{t\in \R}\{Z_t=X\}$. Then
    \[
    \zeta_{ess}=\lim_{n}\frac{\mu_{max}(\pi_*\mathcal L^{\otimes n})}{n}(=\lim_{n}\frac{\mu_{max}({\rm Sym}^n E)}{n}). 
    \]
    where $\mu_{\max}(-)$ is the maximal slope appears in the Harder-Narasimhan filtration of $\pi_*\mathcal L^{\otimes n}$. 

\begin{warn}
    In positive characteristic, the symmetric power of a semistable vector bundle may not remain semistable, making the left-hand side difficult to compute. This issue arises only in positive characteristic. However, this problem can be resolved by sufficiently twisting the Frobenius morphism multiple times.
\end{warn}

\begin{prop}[Ramanan, Ramanathan, \cite{Ramanan1985ProjectiveNO}]
    For above $E$ semistable, there exists $N$ large enough such that for $n\geq N$, any symmetric power of ${\rm Fr}^{n,*}_CE$ is semistable, where ${\rm Fr}_C$ is the absolute Frobenius of $C$. 
\end{prop}

\begin{thm}[Baby-version]
    Fix a sufficiently large $n$ in the sense of the above proposition. Denote $V'={\rm Fr}_K^{n,*}V$ and $\{V'_i\}$ the generic fiber of Harder-Narasimhan filtration $\{E'_i\}$ of ${\rm Fr}_C^{n,*}E$. 
    \begin{enumerate}
        \item The height filtration of $\mathbb P(V')$ is given by $\mathbb{P}(V') \supsetneq \mathbb{P}(V'/V'_1) \supsetneq \mathbb{P}(V'/V'_2) \cdots \supsetneq \mathbb{P}(V'/V'_r)=\emptyset$ with successive minima $\mu_{\rm max}(E'_{i}/E'_{i-1})$. 
        \item The height filtration of $h_{\mathcal O(1)}$ on $\mathbb P(V)$ is the image of the height filtration of $h_{\mathcal O(1)}$ on $\mathbb P(V')$ along the natural map. The successive minima of $\mathbb P(V)$ are $\frac{1}{p^n}\mu_{\rm max}(E'_{i}/E'_{i-1})$. 
        
    \end{enumerate}
\end{thm}

A byproduct would be the following equality: 
\[
L_{max}=\lim_{n}\frac{\mu_{max}({\rm Sym}^n E)}{n}, 
\]
where $L_{max}: = \max_{f:Y \to C}\{\mu_{max}(f^*E)\}$, where $f$ runs through all finite non-constant morphisms. It would be in particular interesting to have a purely vector bundle theoretical proof of this equality. Note that $L_{max}=\mu_{max}$ in characteristic zero, and $L_{max}=\max_{n}\{\mu_{max}(Fr_C^{n,*}E)\}$ in positive characteristic. 

\subsection{Acknowledgments}
We are grateful to Dr. Binggang Qu for his invaluable advice and support during the 2024 Algebra and Number Theory Summer School at Peking University. We also thank the organizers for bringing participants together and for their warm hospitality. The first author additionally thanks Prof. Jiu-Kang Yu for thoughtful comments. This work was largely written during the first author’s graduate study at The Chinese University of Hong Kong.

\section{Strongly canonical reduction}

\subsection{Vector bundles}
Let $E$ be a vector bundle on a projective smooth curve $C$. The \textit{degree} of $E$ is $\deg(E) := \deg(\det(E))$ and the \textit{slope} of $E$ is $\mu(E) := \deg(E)/\operatorname{rk}(E)$. It is called \textit{slope semistable} if for every subbundle $F$ of $E$, $\mu(F) \leq \mu(E)$. This is equivalent to $\mu(Q) \geq \mu(E)$ for every quotient bundle $Q$ of $E$.

There exists uniquely a filtration of subbundles $0=E_0 \subseteq E_1 \subseteq \cdots \subseteq E_r=E$ such that
	\begin{enumerate}
		\item $E_i/E_{i-1}$ is semistable;
		\item $\mu(E_i/E_{i-1}) > \mu(E_{i+1}/E_{i})$.
	\end{enumerate}
This filtration is called the \emph{Harder-Narasimhan filtration} of $E$. $\mu(E_1)$ is called \emph{maximal slope} of $E$ and is also denoted by $\mu_{\max}(E)$. This theory can be generalized for any reductive group $G$, whereas in this case $G=GL_n$ by taking frame bundles of vector bundles of rank $n$. 

\subsection{Principal bundles}
For any linear algebraic group $\Gamma/k$, let $X(\Gamma)=\mathrm{Hom}(\Gamma,\mathbb{G}_m)$ denote the character group of $\Gamma$. For a cocharacter $f \in X(\Gamma)^\vee$ and a character $\lambda \in X(\Gamma)$, we shall denote the pairing by $\langle f, \lambda \rangle$.
For any $\lambda \in X(\Gamma)$, denote by $k_{\lambda}$ the one-dimensional representation on the vector space $k$ with $\Gamma$ acting by $\lambda$.

A \emph{principal $\Gamma$-bundle} on $C$ is a $C$-scheme $F$ equipped with a right action of $\Gamma$ and a $\Gamma$-equivariant smooth morphism $F \longrightarrow C$ such that the map
 \begin{flalign*}
\quad\quad\quad &F \times_C (C \times \Gamma) \longrightarrow F \times_C F,\quad (f,(x,g)) \longmapsto (f,fg) &
 \end{flalign*} is an isomorphism. 

Attached to a principal $\Gamma$-bundle $F$, one has an associated cocharacter
\begin{flalign*}
 \quad\quad\quad & \deg(F): X(\Gamma) \longrightarrow \mathbb{Z},\quad \lambda \longmapsto \langle \deg(F),\lambda \rangle=\deg(F \times_\Gamma k_\lambda),&
\end{flalign*} where 
$\deg(F \times_\Gamma k_\lambda)$ is the degree of the line bundle $F\times_\Gamma k_\lambda$ on the curve $C$.

Let $H$ be a closed subgroup of a linear algebraic group $\Gamma$ over $k$. A \emph{reduction of structure group} of $F$ to $H$ is a pair $(F_H,\phi)$ where $F_H$ is a principal $H$-bundle and $\phi: F_H \times_H \Gamma \simeq F$ is an isomorphism.
The quotient $F/H$ parametrizes all reductions of $F$ to $H$, see \cite{SGA3}. In particular, the assignment sending a section $\sigma: C \longrightarrow F/H$ to the reduction $\sigma^*F$ of $F$ to $H$ is a one-one correspondence between sections of $F/H \longrightarrow C$ and reductions of structure group of $F$ to $H$. 

\subsection{Reductive groups, characters and cocharacters}
Let $G$ be a connected reductive group over $k$. Fix a Borel subgroup $B\subseteq G$ and a maximal torus $T\subseteq B$. Let $W$ be the Weyl group and $\Delta$ be the set of simple roots with respect to $(G,B,T)$. For any $\alpha\in\Delta$, we denote by $\alpha^\vee$ the corresponding simple coroot.

We shall consider only parabolic subgroups containing $B$. For such a parabolic subgroup $P$, let $W_P\subseteq W$ be the Weyl group $W(L_P)$ of the Levi factor $L_P\subseteq P$ and $\Delta_P\subseteq \Delta$ be the simple roots of $L_P$. Note that the natural inclusion 
\begin{flalign*}
	\quad \quad \quad &
	X(P) \longrightarrow X(L_P) \longrightarrow X(Z(L_P))
	&&
\end{flalign*} becomes an isomorphism after tensoring with $\mathbb{Q}$. Thus we have
\begin{flalign*}
\quad\quad\quad & X(T)_\mathbb{Q} \longrightarrow X(Z(L_P))_\mathbb{Q}=X(P)_\mathbb{Q} &
\end{flalign*}
and by taking duals, we get the so-called \emph{slope map} $X(P)^\vee_\mathbb{Q} \longrightarrow X(T)^\vee_\mathbb{Q}$ introduced in \cite[\textsection 2.1.3]{Schieder2015The}.  

\subsection{The strongly canonical reduction}
Let $G$ be a connected reductive group over $k$ and $F$ be a principal $G$-bundle over $C$.
We denote $F(G):=F\times_{G,{\rm int}} G$ the group scheme associated to $F$ by the action of $G$ on itself by inner automorphisms. 
Let $F_P$ be a reduction of $F$ to a parabolic subgroup $P$. Let $\deg(F_P) \in X(T)_\mathbb{Q}^\vee$ be the induced (rational) cocharacter. We also denote $F_P(P)=F_P\times_{P, {\rm int}}P$ the associated group scheme to $F_P$. 

The $G$-bundle $F$ is called \emph{semistable} if for any parabolic subgroup $P$, any reduction $F_P$ of $F$ to $P$ and any dominant character $\lambda$ of $P$ which is trivial on $Z(G)$, we have $\langle \deg(F_P),\lambda \rangle \leq 0$, or equivalently, if any parabolic subgroup scheme of $F(G)$ has non positive-degree. 

\begin{defn}
   We say $F$ is strongly semistable if for any non-constant finite morphism $f: C' \to C$, the pullback $f^*F$ is semistable. 
\end{defn}

\begin{rem}
    Semistability is preserved by separable pullbacks, that is, for a separable morphism $f: C' \to C$, the pullback $f^*(F_P)$ is semistable. However, this fact does not hold for inseparable morphism. 
\end{rem}

Among all filtrations of a vector bundle, there is a canonical one (the Harder-Narasimhan filtration). More generally, among all reduction of a principal $G$-bundles to parabolic subgroups, there is also a canonical one. The reduction $E_P$ of $E$ is called \emph{canonical} 
(or the Harder--Narasimhan filtration) if it satisfies 
the following conditions:

A reduction $F_P$ of $F$ to a parabolic subgroup $P$ is called \emph{canonical} if the following two conditions hold:
	\begin{enumerate}
		\item The the principal $L_P$ bundle $F_P \times_P L_P$ is semistable.
		\item For any non-trivial character $\lambda$ of $P$ which is non-negative linear combination of simple roots, $\langle \deg(F_P), \lambda \rangle>0$.
	\end{enumerate}

\begin{defn}\label{strongly}
        We say $F_P$ is strongly canonical if for any non-constant finite morphism $f: C' \to C$, the pullback $f^*(F_P)$ is canonical.
\end{defn}

\begin{rem}
Similarly as strongly Semistability, the notion of canonical reduction is preserved by separable pullbacks \cite{behrend2000semi}. In characteristic zero, canonical reduction is strong while this does not hold for inseparable morphism, which makes Definition \ref{strongly} necessary. 
\end{rem}

We will need the following theorem. 
\begin{thm}[\cite{1984RR}, Theorem 3.23]\label{RR}
    If $F$ is a strongly semistable $G$-bundle and $\rho :G\xrightarrow[]{} H$ is a homomorphism that maps the connected component of of the center of $G$ into that of $H$, then $F\times^{G}H$ is a semistable $H$-bundle. 
\end{thm}

\section{Height filtrations and successive minima} \label{section_ ht fil and succ min}

Let $C$ be a curve over a field $k$ of positive characteristic and $K$ be its function field. Let $G$ be a connected reductive group over $k$. Let $P \subseteq G$ be a parabolic subgroup and $\lambda: P \longrightarrow \mathbb{G}_m$ be a strictly antidominant character. Let $F$ be a principal $G$-bundle over $C$. Let $\mathcal{X}=F/P$ and $X=(F/P)_K$. Let $\mathcal{L}_\lambda = F \times_P k_\lambda$ and $h_{\mathcal{L}_\lambda}: X(\overline{K}) \longrightarrow \mathbb{R}$ be the induced height function.

\[ \xymatrix{
X=(F/P)_K \ar[d] \ar[r] & \mathcal{X}=F/P \ar[d] \\
\operatorname{Spec}(K) \ar[r] & C
} \]

Let $F_Q$ be the canonical reduction of $F$ to $Q$, and $\deg(F_Q) \in X(T)^\vee_\mathbb{Q}$ be the degree cocharacter. Let $W,W_P$ and $W_Q$ be the Weyl group of $G,L_P$ and $L_Q$. For $w \in W_Q \backslash W/W_P$, let $C_w=(F_Q \times_Q QwP/P)_K \subseteq X$ be the corresponding Schubert cell in $X$. We have the following theorem. 

\begin{thm}\label{charp}
Assume $F$ admits strongly canonical reduction. For any $t \in \mathbb{R}$, let $Z_t \subseteq X$ be the Zariski closure of the set $\big\{ x \in X(\overline{K}): h_{\mathcal{L}_\lambda}(x) < t \big\}$. Then
	\begin{flalign*}
    \quad \quad \quad &
    Z_t = X \Bigg\backslash \coprod_{\langle \deg (F_{Q }), w\lambda \rangle\geq t} C_w = \coprod_{\langle \deg (F_{Q }), w\lambda \rangle<t} C_w.
    &
  \end{flalign*} \label{thm of ht filtration}
In particular, the successive minima are $\zeta_{w}=\left\langle\operatorname{deg}\left(F_{Q}\right), w \lambda\right\rangle$ and Zhang's successive minima are $e_{i}=\min \left\{\zeta_{w}: \ell(w)=\operatorname{dim} G / P-i+1\right\}$ where $$\ell(w)=\max _{\sigma \in W_{Q} w W_{P}} \min _{\tau \in \sigma W_{P}} \ell(\tau).$$
\end{thm}

For general $G$-bundle $F$ which does not necessarily admits a strongly canonical reduction, consider the following Cartesian diagram for $n$ sufficiently large. 
 \[\begin{tikzcd}
	(\operatorname{Fr_C}^n)^* F /P & F/P \\
	C & C
	\arrow["\phi",from=1-1, to=1-2]
	\arrow[from=1-1, to=2-1]
	\arrow[from=1-2, to=2-2]
	\arrow["\operatorname{Fr}_C^n" ,from=2-1, to=2-2]
\end{tikzcd}\]
with its generic fiber 
 \[\begin{tikzcd}
	\Tilde{X}=((\operatorname{Fr_K}^n)^* F /P)_K & X=(F/P)_K \\
	\spec K & \spec K
	\arrow["\phi_K",from=1-1, to=1-2]
	\arrow[from=1-1, to=2-1]
	\arrow[from=1-2, to=2-2]
	\arrow["\operatorname{Fr}_C^n" ,from=2-1, to=2-2]
\end{tikzcd}\]

\begin{thm}\label{twist}
    The height filtration of $h_{\mathcal{L}_{\lambda}}$ on $X$ is given by the image of the height filtration of the height filtration of $\Tilde{X}$, the successive minima are $\frac{1}{p^n}$ of the successive minima of $\Tilde{X}$.

\end{thm}

\subsection{A height lower bound in Schubert cells}\label{3.1}

In subsection \ref{3.1}, we follow \cite{fan2024arakelovgeometryflagvarieties} since there is no difference between characteristic zero and characteristic $p$. For $w \in W_{Q} \backslash W / W_{P}$, write $\mathcal{C}_{w}=F_{Q} \times_{Q} Q w P / P, \mathcal{X}_{w}=F_{Q} \times_{Q} \overline{Q w P} / P, C_{w}=\mathcal{C}_{w, K}$ and $X_{w}=\mathcal{X}_{w, K}$.

Let $F_Q$ be the canonical reduction of $F$ with canonical parabolic $Q$, and $\deg(F_Q) \in X(T)^\vee_\mathbb{Q}$ be the degree cocharacter. Let $W,W_P$ and $W_Q$ be the Weyl groups of $G,L_P$ and $L_Q$. For $w \in W_{Q } \backslash W / W_{P }$, write $\mathcal{C}_w=F_{Q } \times_{Q } QwP/P$, $\mathcal{X}_w=F_{Q } \times_{Q } \overline{QwP}/P$, $C_w=\mathcal{C}_{w,K}$ and $X_w = \mathcal{X}_{w,K}$.

Note that for any $w^{\prime} \in W_{Q}$ and $\lambda \in X(T)$, we have $w^{\prime} \lambda-\lambda \in \mathbb{Z}\left[\Delta_{Q}\right]$ and consequently $\left\langle\operatorname{deg}\left(F_{Q}\right), w^{\prime} \lambda\right\rangle=\left\langle\operatorname{deg}\left(F_{Q}\right), \lambda\right\rangle$. Note also that for any $\lambda \in X(P)$ and $w \in W_{P}, w \lambda=\lambda$. So, the number $\left\langle\operatorname{deg}\left(F_{Q}\right), w \lambda\right\rangle$ is well-defined for any $\lambda \in X(P)$ and $w \in W_{Q} \backslash W / W_{P}$.

\begin{prop}
    For any $x \in C_{w}(\bar{K}), h_{\mathcal{L}_{\lambda}}(x) \geq\left\langle\operatorname{deg} F_{Q}, w \lambda\right\rangle$.
\end{prop}

We follow the argument in \cite{fan2024arakelovgeometryflagvarieties}, since many techniques do not require the characteristic.

For $  x \in X(K) $, let $ \sigma_x : C \to F/P $ be the section induced by $ x : \operatorname{Spec}(K) \to X $ through the valuative criterion. Let $ F_{P,x} = \sigma_x^*F $ be the corresponding reduction to $ P $.

\begin{defn}
    A reduction $ F_P $ of $ F $ to $ P $ is called \textit{in relative position} $ w \in W_Q \setminus W/W_P $ with respect to $ F_Q $ if the image of the natural map
\[
F_Q \times_C F_P \to G_C, \quad (a,b) \mapsto a^{-1}b
\]
lies in $ QwP_C $.

This definition coincides with the definition in \cite[$\S$4.1]{Schieder2015The}.
\end{defn}

The following lemmas give  the height of a rational point and the relation between field extensions.

\begin{lem}
    \begin{enumerate}
        \item[(1)]  For $ x \in X(K) $, $ h_{\mathcal{L}_\lambda}(x) = \langle 
        \deg(F_{P,x}), \lambda \rangle $.

        \item[(2)] If $x \in C_w(K)$, $F_{P,x}$ is of relative position $w$ with respect to $F_Q$.

        \item[(3)] If $F_P$ is in relative position $w$ with respect to $F_Q$, then $ \langle \deg(F_P), \lambda \rangle \geq \langle \deg(F_Q), w\lambda \rangle $ for any antidominant character $ \lambda $.
    \end{enumerate}
\end{lem}

\begin{proof}
    \begin{enumerate}
        \item[(1)] By definition, $ h_{\mathcal{L}_\lambda}(x) $ is the degree of $ \sigma_x^*(\mathcal{L}_\lambda) $ and $ \langle \deg(F_{P,x}), \lambda \rangle $ is the degree of $ F_{P,x} \times_P k_\lambda $. 
        Then it follows from the equality
\[
\sigma_x^*(\mathcal{L}_\lambda) = \sigma_x^*(F \times_P k_\lambda) = (\sigma_x^*F) \times_P k_\lambda = F_{P,x} \times_P k_\lambda
\]

\item[(2)] We have a commutative diagram\[
\begin{array}{ccc}
\sigma_x^* F & \longrightarrow & F = \coprod F_Q \times_Q QwP \\
\downarrow & & \downarrow \\
C & \stackrel{\sigma_x}{\longrightarrow} & F/P = \coprod F_Q \times_Q QwP/P.
\end{array}
\] So, the image lies in $F_Q \times_Q QwP$. The map $F_Q \times_C F_{P,x} \to C \times G$ factors as 
$ F_Q \times_C F_{P,x} \to F_Q \times_C (F_Q \times_Q QwP)\to (F_Q \times_C F_Q) \times_Q QwP $ and $(F_Q \times_C F_Q) \times_Q QwP \to (C \times_Q Q) \times_Q QwP = C \times_Q QwP$.

\item[(3)] See \cite[Theorem4.1]{Schieder2015The}.  
\end{enumerate}
\end{proof}

\begin{proof}[Proof of proposition 3.3]
For any finite extension $L/K$ and $x\in C_w(L)$, we define 
\begin{itemize}
    \item $C_L$: normalization of $C$ in $L$
    \item $X_L=X\times_K L$
    \item $F_{C_L}=F\times_C C_L$
    \item $\Tilde{x}$: the corresponding morphism $\spec L\to X_L$, and define $\sigma_{\Tilde{x}}$ similarly.
    \item $\mathcal{L}_{C_L,\lambda}$: the pullback of $\mathcal{L}_\lambda$.

\end{itemize}

By the above lemma(1), \[
    h_{\mathcal{L}_\lambda}(x)=\frac{1}{[L:K]}\deg \sigma_{\Tilde{x}}^*\mathcal{L}_\lambda=\left\langle \deg(F_{C_L,P,\Tilde{x}}),\lambda \right\rangle 
    \]

    Since $F$ admits strongly canonical reduction, the (strongly) canonical reduction of $F_{C_L}$ is $F_{C_L,Q}=F_Q\times_C C_L$ and $\deg(F_{C_L,Q})=[L:K]\deg (F_Q)$ in $X(T)^{\vee}_{\mathbb{Q}}$. Then the proposition follows from the above lemma (2)(3).
\end{proof}

\subsection{Height filtration and successive minima}
For any $t \in \mathbb{R}$, let $Z_t \subseteq X$ be the Zariski closure of the set $\big\{ x \in X(\overline{K}): h_{\mathcal{L}_\lambda}(x)<t \big\}$. Let $\zeta_{\operatorname{ess}}(X) := \inf \big\{ t: Z_t = X \big\}$ be the \emph{essential minimum} of $h_{\mathcal{L}_\lambda}$ on $X$.

Note that $\mathcal{X}_w$ is a closed subscheme of $\mathcal{X}$, so $\mathcal{L}_\lambda|_{\mathcal{X}_w}$ induces a height function $h_{\mathcal{L}_\lambda}: X_w(\overline{K}) \longrightarrow \mathbb{R}$, which is nothing but the restriction of $h_{\mathcal{L}_\lambda}: X(\overline{K}) \longrightarrow \mathbb{R}$ to $X_w(\overline{K})$. Let $\zeta_{\operatorname{ess}}(X_w)$ be the essential minimum of $h_{\mathcal{L}_\lambda}$ on $X_w$.
Now we compute the \textit{essential minimum} $\zeta_{\rm ess}(X_w)$ of $X_w$. 
\begin{lem}[\cite{jantzen1996representations}, Part I, §5.18]
    On $\mathcal{X}_{w}=F_{Q} \times_{Q} \overline{Q w P} / P \subseteq \mathcal{X}$, we have $$\pi_{*}\left(\left.\mathcal{L}_{\lambda}\right|_{\mathcal{X}_{w}}\right)=F_{Q} \times_{Q} \mathrm{H}^{0}\left(\overline{Q w P} / P, M_{\lambda}\right).$$ 
\end{lem}

Let $V$ be any $Q$-representation and let $V=\bigoplus_\nu V[\nu]$ be its weight decomposition. We define a filtration $V_{\bullet}$ on the vector space $V$:
For any rational number $q \in \mathbb{Q}$, we define the subspace $V_q$ as the sum of weight spaces
$$
V_q:=\bigoplus_{\left\langle\deg F_{Q}, \nu\right\rangle \geq q} V[\nu]
$$
Clearly, $V_{q^{\prime}} \subseteq V_q$ whenever $q^{\prime} \geq q$. We will consider subspaces $V_q$ only for the finitely many $q \in \mathbb{Q}$ where a jump occurs, that is, only for those $q$ such that $V_{q^{\prime}} \subsetneq V_q$ for all $q^{\prime}>q$. Let $q_0$ be the smallest and $q_1$ the largest rational number occurring among such $q$. Then $V_{q_1}$ is the smallest non-zero filtration step and $V_{q_0}$ equals $V$.

Then, by twisting the $Q$-subrepresentations $V_q$ above by $F_Q$, we obtain a filtration $V_{\bullet, F_Q}:=F_Q \times_{Q} V_\bullet$ of the vector bundle $V_{F_Q} =F_{Q}\times_{Q}V$ by subbundles

$$
0 \neq V_{q_1, F_Q} \subsetneq \cdots \subsetneq V_{q, F_Q} \subsetneq \cdots \subsetneq V_{q_0, F_Q}=V_{F_Q}.
$$

   Improving a result of Schieder\footnote{The case of characteristic \(0\) and a specific case of positive characteristic were proven in \cite{Schieder2015The}}, we proved that:

   \begin{prop}\label{HN}
       Assume the redution $F_Q$ is the strongly\footnote{Proposition \ref{HN} fails to hold if we do not require the notion of strongly canonical reduction.} canonical reduction, then the  filtration $V_{\bullet, F_Q}$ of the vector bundle $V_{F_G}$ is the Harder-Narasimhan filtration of $V_{F_G}$. 
   \end{prop}

\begin{proof}[Proof of Proposition \ref{HN}]
    It will suffice to check two assertions: all graded pieces are semistable vector bundles and their slopes are strictly decreasing. 

    As shown in \cite[Lemma 5.1]{Schieder2015The}, first on $\mathrm{gr}_q V_\bullet$ the unipotent radical $U(L_P)$ acts trivially, therefore $\mathrm{gr}_q V_\bullet$ is a $L_Q$-representation. Furthermore, we have
    \[
    \mathrm{gr}_q(V_{\bullet, F_Q})= V_{\bullet,F_Q}/V_{\bullet-1,F_Q}=(V_{\bullet}/V_{\bullet-1})_{F_Q}=(\mathrm{gr}_q V_\bullet)_{F_{L_Q}}.
    \]
    If one writes the characters that appears in $V_q$ as $\sum{m_i} \lambda_i$, then its slope is computed as in
    \[
    \mu((\mathrm{gr}_q V_\bullet)_{F_{L_Q}})=\frac{\sum\langle\deg F_Q,\lambda_i\rangle m_i}{\sum m_i}=q. 
    \]
    Therefore, the slopes of the graded pieces are decreasing by definition. 

    To prove the semistability, consider the semisimplification $W$ of $\mathrm{gr}_q V_\bullet$ as a representation of ${L_Q}$. From the same computation, we see that for each simple direct summand $W_i$ of $W$ the associated bundle $W_{i, F_{L_Q}}$ has slope $q$. Thus $(\mathrm{gr}_q V_\bullet)_{F_{L_Q}}$ admits a filtration by subbundles such that each graded piece is a semistable vector bundle of slope $q$, where the semistability is obtained by applying Theorem \ref{RR} to ${L_Q} \to \gl(W_i)$. Then $(\mathrm{gr}_q V_\bullet)_{F_M}$ is also semistable of slope $q$ by the following. 
\end{proof}

\begin{lem}
    If $E$ is a vector bundle on a curve $C$ that admits a filtration of subbundles 
    \[
    0 \subsetneq E_1 \subsetneq \cdots \subsetneq E_r=E, 
    \]
    and all graded pieces are semistable of slope equals to a given number $\lambda$, then $E$ is also semistable of slope equals to $\lambda$. 
\end{lem}

\begin{proof}
    The proof is elementary. See, for example \cite[Page 580]{2023Algebraic}
\end{proof}

\begin{cor}
    Assume $F_Q$ is the strongly canonical reduction of $F$. The Harder-Narasimhan filtration of $F_{Q} \times_{Q} H^{0}\left(\overline{Q w P} / P, M_{\lambda}\right)$ is $F_{Q} \times_{Q} H^{0}\left(\overline{Q w P} / P, M_{\lambda}\right)_{\bullet, \operatorname{deg}\left(F_{Q}\right)}$. Moreover, the maximal slope is $\left\langle\operatorname{deg}\left(F_{Q}\right), w \lambda\right\rangle$.
\end{cor}

\begin{proof}
    The first assertion is obtained by taking $V$ to be $H^{0}\left(\overline{Q w P} / P, M_{\lambda}\right)$ in Proposition \ref{HN}. For the second assertion, it will suffice to show the highest weights in $H^{0}\left(\overline{Q w P} / P, M_{\lambda}\right)$ belong to $w\lambda+\Z[\Delta_Q]$. The argument in \cite{fan2024arakelovgeometryflagvarieties} is also valid without any assumption on the characteristic of base field. Basically, one can use the restriction map 
    \[
    \mathrm{H}^0(\overline{QwP}/P, M_\lambda) \to \mathrm{H}^0(\overline{Bw'P}/P, M_\lambda)
    \]
    is surjective for all $w'$ lying over the fibre of $w$ along $W/W_P \to W_Q\backslash W/W_P$, and the diagonal map 
    \[
    \mathrm{H}^0(\overline{QwP}/P, M_\lambda) \to \prod_{w'} \mathrm{H}^0(\overline{Bw'P}/P, M_\lambda)
    \]
    is injective. These show that the highest weights appears and only appears in that of $\mathrm{H}^0(\overline{Bw'P}/P, M_\lambda)$ for all $w'$, and they ahve the unique highest weights $w'\lambda\in \lambda+\Z[\Delta_Q]$. 
\end{proof}

\begin{cor}
    Under the assumption that $F$ has strongly canonical reduction, the essential minimum $\zeta_1(h_{\mathcal L_\lambda}, X_w)$ of $h_{\mathcal L_\lambda}$ on $X_w$ is $\langle \deg(F_Q), w\lambda\rangle$. 
\end{cor}

\begin{proof}
    By Balla\"y's theorem \cite[2]{Ballay21}, 
    \[
    \zeta_1(h_{\mathcal L_\lambda}, X_w)=\lim_{n\to \infty} \frac{\mu_{\rm max}(\pi_*\mathcal L_{n\lambda}|_{X_w})}{n}=\lim_{n\to \infty} \frac{n\langle \deg(F_Q), w\lambda\rangle}{n}= \langle \deg(F_Q), w\lambda\rangle. 
    \]
\end{proof}

\begin{proof}[Proof of Theorem \ref{charp}]
    The arguments follows from the fact that the lower bound and essential minimum are invariant under taking closed subvariety as in \cite[Proof of Theorem 2.1]{fan2024arakelovgeometryflagvarieties}. 
\end{proof}

Now we treat the case $F$ may not have a strongly canonical reduction. 

\begin{proof}[Proof of Theorem \ref{twist}]
    Theorem 5.1 of \cite{Langer05} shows for a principal $G$-bundle $F$ (might not admits a strongly canonical reduction), the strongly canonical reduction exists for $({\rm Fr}^n)^* F$, when $n$ is sufficiently large. We have the following cartesian diagram
 \[\begin{tikzcd}
	(\operatorname{Fr_C}^n)^* F /P & F/P \\
	C & C
	\arrow["\phi",from=1-1, to=1-2]
	\arrow[from=1-1, to=2-1]
	\arrow[from=1-2, to=2-2]
	\arrow["\operatorname{Fr}^n" ,from=2-1, to=2-2]
\end{tikzcd}\]
with its generic fibre
 \[\begin{tikzcd}
	\Tilde{X}=((\operatorname{Fr_K}^n)^* F /P)_K & X=(F/P)_K \\
	\spec K & \spec K
	\arrow["\phi_K",from=1-1, to=1-2]
	\arrow[from=1-1, to=2-1]
	\arrow[from=1-2, to=2-2]
	\arrow["\operatorname{Fr}^n" ,from=2-1, to=2-2]
\end{tikzcd}\]
Here $\operatorname{Fr_C}$ is the absolute Frobenius on $C$. Suppose $n$ is large enough such that $(\operatorname{Fr}^n)^* F$ has strongly canonical reduction. Note that $\phi$ is purely inseparable and we have the equality
\[
\frac{1}{p^n}h_{\phi^*\mathcal L_\lambda}(x)=h_{\mathcal L_\lambda}(\phi_K(x))
\]
for all $x\in \Tilde{X}$. Our theorem follows. 
\end{proof}

\newpage
\bibliography{ref}
\bibliographystyle{plain}
\end{document}